\documentclass[reqno]{amsart}
\usepackage{amssymb}
\usepackage[usenames,dvipsnames]{xcolor}

\newtheorem{defn}{Definition}[section]
\newtheorem{lem}[defn]{Lemma}
\newtheorem{ex}[defn]{Example}

\newtheorem{rem1}[defn]{Remark}
\newtheorem{thm}[defn]{Theorem}

\newtheorem{cor}[defn]{Corollary}

\newcommand{\N}{\mathbb N}
\newcommand{\R}{\mathbb R}

\newcommand{\di}{\displaystyle}

\newcommand{\FF}{\mathcal{F}}
\newcommand{\diam}{{\rm diam}}
\newcommand{\F}{$\mathcal{F}$}
\newcommand{\x}{$\boldmath x$}

\begin{document}

\title[Countable Contraction Maps]{Countable contraction maps in metric spaces: Invariant Sets and Measures}


\subjclass[2000]{Primary 28A80; Secondary 37C25; 37C70}

\keywords{Countable iterated function system, Invariant set, Invariant measure}

\author{Mar\'{\i}a Fernanda Barrozo}
\address{Departamento de
Matem\'atica \\ Facultad de Ciencias F\'isico-Matemáticas y Naturales\\
Universidad Nacional de San Luis\\ and IMASL, CONICET, Argentina}
\email[Mar\'{\i}a Fernanda Barrozo]{mfbarroz@unsl.edu.ar}

\author{Ursula Molter}
\address{Departamento de
Matem\'atica \\ Facultad de Ciencias Exactas y Naturales\\ Universidad
de Buenos Aires\\ Ciudad Universitaria, Pabell\'on I\\ 1428 Capital
Federal\\ ARGENTINA\\ and IMAS, CONICET, Argentina}
\email[Ursula~M.~Molter]{umolter@dm.uba.ar}
\thanks{The authors acknowledge
support from the following grants: PICT 2011-0436 (ANPCyT), PIP
01070 and 2008-398 (CONICET), UBACyT 20020100100502 (UBA) and PROICO 3-0412
(UNSL) }

\date{\today}

\begin{abstract}
We consider a complete metric space $(X,d)$ and a countable number
of contractive mappings on $X$, $\FF=\{F_i:i\in\N\}$. We show the
existence of a {\em smallest} invariant set (with respect to
inclusion) for $\FF$. If the maps $F_i$ are of the form $F_i(\x) = r_i \x + b_i$ on $X=\R^d$, we can prove a converse of the classic result on contraction maps. Precisely, we can show that for that case, there exists a {\em unique} bounded invariant set if and only if $r = \sup_i r_i$ is strictly smaller than $1$.

 Further, if $\rho = \{\rho_k\}_{k\in \N}$ is a
probability sequence, we show that if there exists an invariant
measure for the system $(\FF,\rho)$, then it's support must be
precisely this smallest invariant set. If in addition there exists any {\em bounded} invariant set, this
invariant measure is unique - even though there may be more than one invariant set.
\end{abstract}

\maketitle

\section{Introduction}

A map $F$ from a metric space $(X,d)$ into itself is a
\textbf{contraction}, if there exists a constant $c$, $0 < c <1$, such that $d(F(x),
F(y)) \leq c d(x,y)$, for all $x, y \in X$. We denote by
$\text{Lip}(F)$ the smallest of all such constants and we call it
\emph{Lipschitz constant} or \emph{contraction factor} of $F$.

In \cite{Hut81}, Hutchinson introduced the notion of {\em invariant
set} and {\em invariant measure} for a finite set of contraction
mappings from a complete metric space $(X,d)$ into itself. In
particular in that paper he proves the now {\em classical} results:

\begin{thm} \label{h1} \cite{Hut81} Let $\mathcal{F}=\{F_1,\dots,F_N\}$ be a finite family of contraction maps in the complete metric space $(X,d)$. Let
$\mathcal{C}(X)$ be the set of non-empty closed and bounded subsets
of $X$, and let $\rho_1,\dots,\rho_N \in (0,1)$ and
$\sum_{i=1}^N\rho_i=1$. Then we have:\begin{itemize}
\item There exists a {\em unique} set $K\in \mathcal{C}(X)$ that is invariant with respect to $\FF$, i.e.
$$K=\overline{\bigcup_{i=1}^NF_i(K)}.$$
$K$ is in fact compact and is the closure of the set of fixed points of all finite compositions of elements of $\mathcal{F}$.
\item There exists a unique Borel regular (outer) measure $\mu$ with bounded support, and total mass $1$, that is invariant with respect to $(\FF,\rho)$, i.e.
$$\mu=\sum_{i=1}^N\rho_i {F_i}_\sharp\mu,$$
where ${F_i}_\sharp\mu$ is the measure defined by ${F_i}_\sharp\mu(E)=\mu(F_i^{-1}(E))$ for each
$E\subset X$.
\item The support of $\mu$, $\text{supp}~\mu$ is precisely the invariant set $K$.
\end{itemize}
\end{thm}

Similar results can be found in
\cite{Bar85}. For general references see \cite{Fal85}, \cite{Fal90}
and \cite{Mat95}.

If instead of a finite number of contraction maps one considers a
countable collection, the notions of invariant set and invariant
measure can be extended in a natural way:

\begin{defn} \label{infinite-invariant} Let $\mathcal{F}=\{F_i\}_{i=1}^\infty$ be a countable family of contractions in the complete metric space $(X,d)$.
We say that a non-empty set $E\subset X$ is an \textbf{invariant
set} for $\FF$ if
$$E=\overline{\bigcup_{i=1}^\infty F_i(E)}.$$ If
$\rho=(\rho_1,\rho_2,\dots)$ is a probability secuence, i.e.
$\rho_i\in(0,1)$ and $\sum_{i=1}^\infty\rho_i=1$, we say that an
outer measure $\mu$ is an \textbf{invariant measure} for
$(\FF,\rho)$ if
$$\mu=\sum_{i=1}^\infty\rho_i {F_i}_\sharp\mu,$$
where ${F_i}_\sharp\mu$ is (as before) the measure defined by
${F_i}_\sharp\mu(E)=\mu(F_i^{-1}(E))$ for each $E\subset X$.
\end{defn}

Finite families of contractive mappings automatically satisfy two
conditions which allow to ensure the existence and uniqueness of a
bounded invariant set: on one hand the boundedness of the set of
fixed points, and on the other hand the fact that the maximum of
Lipschitz constants is {\em strictly} less than $1$. In general, if
one has a countable system these
conditions are not automatically satisfied.

\begin{defn}\label{def:rD} A set
$\FF := \{F_i\}_{i\in I}$, for $I \subset \N$ either finite or infinite, in a complete metric space $(X,d)$, where $F_i$ are contraction maps, will be called {\em Iterated Function System} (IFS). We will denote by $r$ the supremum of  the contraction factors, and by $D$ the set of fixed points, i.e.
 \begin{align*}
 r &:= \sup_{i\in I} \{ r_i : r_i \ \text{contraction factor of} \ F_i\} \quad \text{and}\\
 D &:= \{x_i, i\in I : F_i(x_i) = x_i \}.
 \end{align*}
 \end{defn}

In the case that $D$ is bounded and
$r<1$, Bandt \cite{Ban89} shows the existence and uniqueness of a
\textbf{bounded invariant set} for $\FF$, where
$\FF=\{F_i\}_{i=1}^\infty$ is a countable family of contractive
mappings. Countable iterated function systems were first introduced
by Mauldin and Williams (\cite{MW86}, see also \cite{Mau95} and
\cite{MU96b}).

In the present article we show that for \textbf{any} countable family of
contractive mappings $\FF$ there exists an invariant set (see also
\cite{Sec12}). In fact, there exists a smallest invariant set, with
respect to inclusion, for $\FF$. We show that this set is the
closure of the set of fixed points of finite compositions of members
of $\FF$. For this result we do not need to assume that $D$ is
bounded neither that $r<1$.

It follows that the bounded invariant set obtained
by Bandt is - as in the finite case - the closure of the set of
fixed points of finite compositions of members of $\FF$.

We further show that the boundedness of $D$ is necessary for the
existence of a \textbf{bounded} invariant set. In fact, since any
invariant set contains $D$, if there exists a bounded invariant set
then $D$ is bounded. However we will show that the condition $r<1$ is not necessary for
the existence of a bounded invariant set.

Further, the condition
that $D$ be bounded is not sufficient: we exhibit an example in
which the set of fixed points is bounded, but there does not exist a
bounded invariant set. However, if $X = \R$ and the system only contains
non-increasing functions, the boundedness of $D$ does suffice (c.f.
Theorem \ref{no-decreciente}).

In addition, we prove a kind of converse to the Theorem by Bandt in
\cite{Ban89}: under certain restrictions, if there exists a unique
bounded invariant set, then necessarily $r = \sup_i r_i <1$ (cf.
Theorem~\ref{lineales en R}).

Finally, we prove that the support
of {\bf any} invariant measure for the countable IFS $(\FF,\rho)$,
where $\rho$ is a probability sequence, must coincide precisely with
the smallest invariant set that we showed to exist. We further show, that if there exists a bounded invariant set, then the invariant measure exists and is \textbf{unique}, even though the invariant set might not be unique.

\section {Invariant Sets}

Let $\mathcal{F}=\{F_i:i\in\N\}$ be a countable family of
contractive mappings in the complete metric space $(X,d)$. First, we
will prove that there exists a smallest invariant set for $\FF$,
with respect to inclusion: the closure of the set of fixed points of
finite compositions of members of $\FF$.

As before,  $x_i$ will denote the fixed point of $F_i$
and $r_i$ will be the Lipschitz constant of $F_i$, i.e.,
  $r_i:=Lip(F_i)$.  $F_{i_1\dots i_k}$ will denote the composition $F_{i_1}\circ\dots\circ
  F_{i_k}$ and $x_{i_1\dots i_k}$ will be the fixed point of $F_{i_1\dots i_k}$.
  Further, $P$ will be the set of fixed points of finite compositions of members of $\FF$.

We first need the following result, which is analogous to the finite IFS case:

\begin{lem}\label{b6} Let $\mathcal{F}=\{F_i:i\in\N\}$ be a countable family of
contractive mappings in the complete metric space $(X,d)$ and let
$P$ be the set of fixed points of finite compositions of members of
$\FF$. If $A\subset X$ is a non-empty closed set such that
$F_i(A)\subset A$ for all $i\in\N$, then $P\subset A$.
\end{lem}

\begin{proof}\

Let $x_{i_1\dots i_p}$ be the fixed point of $F_{i_1\dots i_p}$, and
let $a\in A$. 

$$\text{Then}\quad \lim_{k\mapsto\infty} F_{i_1\dots i_p}^k(a) = x_{i_1\dots i_p}.
$$

$$\text{Since} \quad F_i(A)\subset A\quad \text{for all}\quad i\in\N, \quad F_{i_1\dots i_p}(a)\in A,$$
we have $F_{i_1\dots i_p}^k(a)\in A$ for all $k$. 

Since $A$
is closed, $x_{i_1\dots i_p}\in A$. Hence, $P\subset A$.
\end{proof}

\begin{thm} \label{b7} Let $\mathcal{F}=\{F_i:i\in\N\}$ be a countable family of
contractive mappings in the complete metric space $(X,d)$. If $P$ is
the set of fixed points of finite compositions of members of $\FF$,
then $\overline{P}$ is the smallest invariant set for $\FF$, with
respect to inclusion.
\end{thm}

\begin{proof}

First, we will prove that $\overline{P}$ is an invariant set for
$\FF$.

Note that for a fixed $N\in\N$, if we consider the finite sub-family
$\{F_1,\dots,F_N\}$, and let $P_N$ be the set of fixed points of
finite compositions of $F_i$ with $1\leq i\leq N$, from
Theorem \ref{h1} it follows that $\overline{P_N}$ is the unique
compact invariant set for the IFS $\{F_1,\dots,F_N\}$.

In order to prove the inclusion $\overline{\bigcup
F_i(\overline{P})}\subset\overline{P}$ it is enough to show that
$F_i(P)\subset \overline{P}$ for all $i$. For this, let $i\in\N$ be fixed, and let $x_{\alpha_1\dots\alpha_n}\in P$ be the fixed point of
$F_{\alpha_1\dots\alpha_n}$. We need to show that $F_i(x_{\alpha_1\dots\alpha_n}) \in \overline{P}$.

Define the following sequence $\{y_k\}_{k\in \N}$ in
$P$: \\
for each $k$, we let $y_k$ be the fixed point of $F_i\circ F_{\alpha_1\dots\alpha_n}^k$. So
$y_1=x_{i\alpha_1\dots\alpha_n}$, $y_2=x_{i\alpha_1\dots\alpha_n\alpha_1\dots\alpha_n}$; etc.

If $N :=\max\{\alpha_1,\dots,\alpha_n,i\}$, then
$x_{\alpha_1\dots\alpha_n}\in P_N$ and $y_k\in P_N$ for all $k$.
Therefore,
$$d(y_k,F_i(x_{\alpha_1\dots\alpha_n}))\leq
r_id(F_{\alpha_1\dots\alpha_n}^k(y_k),x_{\alpha_1\dots\alpha_n})
 < (r_{\alpha_1}\dots r_{\alpha_n})^k \diam~P_N.$$

Hence, since $r_{\alpha_1}\dots r_{\alpha_n}<1$ and $\diam~P_N<\infty$,
we have
\begin{equation*} \lim_{k\rightarrow\infty}y_k=
F_i(x_{\alpha_1\dots\alpha_n}), \end{equation*}
which implies that $F_i(x_{\alpha_1\dots\alpha_n})\in \overline{P}$, as we wanted to show.

For the other inclusion we will show that $P\subset \bigcup
F_i(\overline{P})$. Let $x_{\alpha_1\dots\alpha_n}\in P$ and
consider the sequence $\{z_k\}_{k\in\N}$, where $z_k$ is the fixed
point of the composition
$F_{\alpha_2\dots\alpha_n}\circ
F_{\alpha_1\dots\alpha_n}^k$. As before, if
$N :=\max\{\alpha_1,\dots,\alpha_n\}$ then
$x_{\alpha_1\dots\alpha_n}\in P_N$ and $z_k\in P_N$ for all $k$.
Consecuently,
\begin{align*}d(z_k,F_{\alpha_2\dots\alpha_n}(x_{\alpha_1\dots\alpha_n}))&\leq
r_{\alpha_2}\dots r_{\alpha_n} (r_{\alpha_1}\dots r_{\alpha_n})^k
d(z_k,x_{\alpha_1\dots\alpha_n})\\
&<(r_{\alpha_1}\dots r_{\alpha_n})^k \diam~P_N,
\end{align*}
which implies
$\di\lim_{k\rightarrow\infty}z_k = F_{\alpha_2\dots\alpha_n}(x_{\alpha_1\dots\alpha_n})
\in \overline{P}$ and therefore
$$x_{\alpha_1\dots\alpha_n}=F_{\alpha_1}(F_{\alpha_2\dots\alpha_n}(x_{\alpha_1\dots\alpha_n}))\in
F_{\alpha_1}(\overline{P}).$$

Thus, the closure of $P$ is an invariant set for $\FF$.

In order to show that $\overline{P}$ is the smallest invariant set,
let $A$ be an invariant set for $\FF$. By Definition~\ref{infinite-invariant}
$A$ is non-empty and closed and satisfies
$F_i(A)\subset A$ for all $i\in\N$. By Lemma~\ref{b6} we obtain that
$P\subset A$, therefore $\overline{P}\subset A$.

\end{proof}

\begin{rem1} Notice that in the previous theorem we do not assume that  $D$ is bounded
neither $r<1$. Therefore, for all countable family of contraction
maps there exists an invariant set.
\end{rem1}

\begin{rem1} The assertion that \emph{$\overline{P}$ is an invariant set}
can also be obtained using Hutchinson's Theorem \ref{h1} and a result
in \cite{Sec12}.
\end{rem1}

Recalling the results of Bandt, \cite{Ban89}, which show that if
$r<1$ and  $D$ is bounded, there exists a unique closed and
\textbf{bounded invariant set} with respect to $\FF$; using
Theorem~\ref{b7}, this unique set must necessarily coincide with
$\overline{P}$ (see also \cite{MM09} and \cite{Hil12}). Note that
Bandt also shows that this set is not necessarily compact. This
extends completely the result of Hutchinson to the countable IFS
case:

\begin{cor}\label{w9} Let $\mathcal{F}=\{F_i:i\in\N\}$ be a countable family of contraction maps in the complete metric space $(X,d)$.
If $r:=\sup_{i\in\N} \text{Lip}(F_i)<1$ and $D$, the set of fixed
points of  elements of $\FF$, is bounded, then $\overline{P}$ is the
unique bounded invariant set for $\mathcal{F}$, where $P$ is the set
of fixed points of finite compositions of members of $\FF$.
\end{cor}

From Theorem \ref{b7} it follows that the boundedness of $D$ is a
necessary condition for the existence of a bounded invariant set.
Indeed, since every invariant set contains $D$, if there exists a
bounded invariant set, then $D$ must be bounded.

On the other hand, the next example shows that the condition $r<1$
is not necessary.

\begin{ex}  For every $i\in\N$ we define $F_i:\R\to\R$,
$$F_i(x)=\left(\frac{i}{i+1}\right)x+\frac{1}{(i+1)^2}.$$
Then, the set of fixed points $D=\{1/(i+1)\}_{i\in\N}$ is bounded,
but $\sup r_i=1$. However, there exists a bounded invariant set, for
example, $[0,1/2]$. Indeed, for every $a\leq 0$ and $b\geq 1/2$, the
closed interval $[a,b]$ is an invariant set.
\end{ex}

The previous example can be extended to every countable family of
non-decreasing contractions in $\mathbb{R}$:

\begin{thm} \label{no-decreciente} Let $\mathcal{F}=\{F_i:i\in\N\}$ be a countable family of contraction maps in $\mathbb{R}$,
such that every $F_i$ is non-decreasing. If the set of fixed
points of members of $\FF$ is bounded, then there exists a bounded
invariant set for $\FF$.
\end{thm}

\begin{proof}

 Let $D$ be the set of fixed points of members of $\FF$, and let us consider $\alpha=\inf D$ and
$\beta=\sup D$. We will show that the interval $I=[\alpha,\beta]$
satisfies $F_i(I)\subset I$ for all $i\in \N$.

Since every $F_i$ is continuous and non-decreasing, we have that
$F_i(I)=[F_i(\alpha), F_i(\beta)]$. Further, since $F_i$ is
contractive, $F_i(\beta)-F_i(x_i)<\beta-x_i$ and
$F_i(x_i)-F_i(\alpha)<x_i-\alpha$. Therefore $\alpha<F_i(\alpha)\leq
F_i(\beta)<\beta$. Thus, $F_i(I)\subset I$.

The conclusion now follows from Lemma \ref{b6}.
\end{proof}

However, in the general case, the boundedness of $D$ is not a
sufficient condition  for the existence of a bounded invariant set, since it can be the case (as we show below) that $D$ is bounded, but $P$ is not.

\begin{ex}  {\rm We consider the contraction maps in $\mathbb{R}$ defined by:}
$$F_i(x)=-\frac{i}{i+1}x+\frac{2i+1}{i}~~~~;~~~~\widetilde{F}_i(x)=-\frac{i}{i+1}x+\frac{1}{i+1}~~~~, i\in\mathbb{N}.$$
\end{ex}
Let $\FF$ be the countable family
$\FF := \{F_i:i\in\N\}\cup\{\widetilde{F}_i:i\in\N\}$. The set of fixed
points is contained in [0,2]. However, there does not exist any bounded
invariant set for \F.

To see this, we consider the compositions $\widetilde{F}_i\circ F_i$
and ${F}_i\circ \widetilde{F}_i$ ($i\in\N$) and look at the set
$\{y_i: \widetilde{F}_i\circ F_i(y_i) = y_i \}\cup\{ z_i: {F}_i\circ
\widetilde{F}_i(z_i)=z_i\}$. A simple computation shows that the
fixed point of $\widetilde{F}_i\circ F_i$ is $y_i =
-\frac{2i(i+1)}{2i+1}$ and the fixed point of ${F}_i\circ
\widetilde{F}_i$ is $z_i = \frac{-i^2+(i+1)^2(2i+1)}{i(2i+1)}$.
Thus, the set $P$ can not be bounded. \hfill $\Box$

We will conclude this section proving a kind of converse to the Theorem by Bandt. In some cases we will be able to prove that, if there exists a unique bounded invariant set, then necessarily $r = \sup_i r_i <1$. We begin by proving a general lemma about enlargements of an invariant set $A$. Recall that the $\varepsilon$-enlargement of a set $A$ in a metric space $(X, d)$ is defined by:
\begin{equation}\label{enlargement}
A_\varepsilon := \{y \in X: d(y,A) < \varepsilon\}.
\end{equation}

\begin{lem} \label{T4} Let $\FF$ be a countable family of contraction maps in a complete metric space $(X, d)$, and let $A$ an invariant set for $\FF$. If $\alpha > 0$, then $F_i(A_\alpha)\subset A_\alpha$ for all $i$.
\end{lem}

\begin{proof}
Let $x\in A_\alpha$. Then, by definition of $A_{\alpha}$, there is
$y\in A$ such that $d(x,y)<\alpha$. Since $A$ is invariant, $F_i(y)\in A$.
Moreover, $d(F_i(x),F_i(y))\leq r_i \cdot d(x,y)<\alpha$ and therefore
$F_i(x)\in A_\alpha$.
\end{proof}

For $X = \R$, and $F_i$ similarities, we can sharpen the previous result:
\begin{thm} \label{lineales en R} Let $\FF$ a countable family of contractive similarities in $\R$ ($|F_i(x)-F_i(y)|= r_i|x-y|$ for all $x, y\in\R$, $r_i < 1$). Let  $r:=\sup r_i = 1$. If there exists a bounded set $A$ that is invariant for $\FF$, then there exists  $\alpha>0$ such that $\overline{A_\alpha}$ is invariant. (Hence if a bounded invariant set exists, it is not unique, in contrast to the case $r < 1$!)
\end{thm}

\begin{proof}

Let $\alpha > 0$ be such that $A_\alpha=(a,b)=I$ for some interval
$I$. By the previous Lemma, we have that $F_i(I)\subset I$ for each
$i$. Then $\overline{\bigcup_{i=1}^\infty F_i(\overline{I})}\subset
\overline{I}$.

For the other inclusion, let $x \in I$ and let $\delta=\max\{b-x,
x-a\}$. Now choose  $r_i>\delta/(b-a)$, which is possible, since
$r=1$. Since each $F_i$ is a similarity,  $F_i(I)$ is an interval
either $(F_i(a), F_i(b))$ or $(F_i(b), F_i(a))$ depending on the
monotonicity of $F_i$.

If $F_i$ is increasing,
$F_i(b)-F_i(a)=r_i(b-a)>\delta$. This implies that
$$F_i(b)-F_i(a)>b-x\quad \text{ and } \quad F_i(b)-F_i(a)>x-a,$$
 and consequently
 $$x-F_i(a)>b-F_i(b)>0\quad \text{ and } \quad F_i(b)-x>F_i(a)-a>0.$$
  Hence
$F_i(a)<x<F_i(b)$ and so $x\in F_i(I)$.

If in turn $F_i$ is decreasing, an analogous reasoning allows us to conclude that, $x\in F_i(I)$.

Hence $\overline{I}\subset\overline{\bigcup_{i=1}^\infty
F_i(\overline{I})}$.

\end{proof}

As a corollary we have a converse to the Theorem of Bandt:
\begin{cor} Let $\FF$ be a countable family of similarities in $\R$. If there exists a {\bf unique} bounded invariant set for $\FF$, then $r < 1$.
\end{cor}

The previous theorem can be extended to similarities in $\R^n$ that
are multiples of the identity, i.e. $F_i:\R^n \rightarrow \R^n$,
$F_i(\mathbf{x})=r_i\mathbf{x}+\mathbf{b}_i$, $|r_i|<1$. Let as
before,  $r:=\sup |r_i|$.

\begin{thm} Let $\FF = \{F_i\}_{i\in \N}$, with $F_i:\R^n \rightarrow \R^n$,
$F_i(\mathbf{x})=r_i\mathbf{x}+\mathbf{b}_i$, $|r_i|<1$. Let as
before,  $r:=\sup |r_i|$, and $P$ be the subset of $\R^n$ of fixed
points of finite compositions of $\{F_i\}_{i\in \N}$. If $P$ is
bounded and $r=1$, then there exists  a rectangle $R = I_1 \times \dots \times I_n$ such that $P
\subsetneq R$ that is invariant for $\FF$.
\end{thm}

\begin{proof}
We write $\mathbf{x}=(x_1,\dots,x_n)$ and
$\mathbf{b}_i=(b_{i1},\dots,b_{in})$, thus
\begin{center}$F_i(\mathbf{x})=r_i\mathbf{x}+\mathbf{b}_i=(r_ix_1+b_{i1},...,r_ix_n+b_{in})$,
\end{center}
and we call $f_{ij}$ the maps from $\R$ to $\R$ defined by the j-th
coordinate ($f_{ij}(x_j)=r_ix_j+b_{ij}$), so
$F_i(\mathbf{x})=(f_{i1}(x_1),...,f_{in}(x_n))$
 and $f_j$ is a
contractive similarity in $\R$, of contraction factor $r_i$, for
$j=1,...,n$.

We have that $\mathbf{x}$ is a fixed point for $F_i$ if and only if
$x_j$ is a fixed point for $f_{ij}$ for $j=1, \dots,n$.

Now, for every $j=1,\dots,n$, let $\FF_j:=\{f_{ij}\}_{i\in \N}$ be
the countable IFS on the line, defined by the "coordinate" maps of
$\FF$. If $P_j$ is the set of fixed points of finite compositions of
the maps from $\FF_j$, by our assumption we have that
$\overline{P}_j$ is bounded and invariant for $\FF_j$ for $j=1,
\dots,n$. By the proof of Theorem
\ref{lineales en R}, there exist intervals $I_j$ such that
\begin{itemize}
\item $\overline{P}_j \subsetneq I_j$, and therefore $P \subsetneq R:=I_1\times\dots\times I_n$.
\item $f_{ij}(I_j)\subset I_j$ for all $i$.
\end{itemize}
Hence, since $F_i = r_i Id  +\mathbf{b}_i$ we have that
$F_i(R)=f_{i1}(I_1)\times\dots\times f_{in}(I_n)\subset
I_1\times\dots\times I_n=R$ for all $i$.

Let now $\mathbf{x}\in R$ (i.e. $x_j \in I_j, 1\leq j \leq n$). If $I_j=[\alpha_j,\beta_j]$ we take
$\delta_j:=\max\{\beta_j-x_j; x_j-\alpha_j\}$. Since $r = 1$, we
choose $i \in \N$ such that
$r_i>\max\{\frac{\delta_j}{\beta_j-\alpha_j}: 1\leq j\leq n\}$.

From the proof of Theorem \ref{lineales en R} it follows that for each $j, 1\leq j \leq n$, $x_j\in
f_{ij}(I_j)$. Hence $\mathbf{x} \in f_{i1}(I_1)\times\dots\times f_{in}(I_n) = F_i(R)$.

So $R$ is invariant for $\FF$.

\end{proof}

We again have the same corollary:
\begin{cor} Let $\FF$ is a countable family of similarities in $\R^n$, such that $F_i(\mathbf{x})=r_i\mathbf{x}+\mathbf{b}_i$, $|r_i|<1$. If there exists a {\bf unique} bounded invariant set for $\FF$, then $r:=\sup |r_i| < 1$.
\end{cor}

\section{Invariant Measures}

In addition to providing a complete proof of Hutchinson's theorem
for the most general case, Kravchenko in \cite{Kra06} generalizes
Hutchinson's theorem \ref{h1} to the case of a countable set of maps
and he gives a sufficient condition for the existence and uniqueness
of an invariant measure. We first need to recall the following
definition.
\begin{defn} A measure $\nu$ is \emph{separable} if there exists a separable Borel set
$A\subset X$ such that $\nu(X\setminus A)=0$. \end{defn}

Note that if $\nu$ is a finite measure, $\nu$ is separable if and
only if $\nu(X\setminus \text{supp}~\nu)=0$.

\begin{thm}\cite{Kra06} \label{b4} Let $(X,d)$ be a complete metric space. Let
$\mathcal{F}=\{F_i:i\in\N\}$ be a countable number of contractions
with fixed points  $x_i$. Let $\rho=\{\rho_i\}_{i\in\N}$, be a probability sequence, i.e. $0<\rho_i<1$
and $\sum_{i=1}^\infty \rho_i=1$. If $\sum_{i=1}^\infty
\rho_i~d(x_1,x_i)<\infty$, then there exists a unique measure
$\mu\in M_s(X)$ that is invariant with respect to $(\FF,\rho)$, i.e.
$$\mu=\sum_{i\in\mathbb{N}}\rho_i {F_i}_\sharp\mu.$$
Here $M_s(X)$ is the space of all separable probability measures
that satisfy $\int f~d\mu<\infty$ for all $f:X\to\R$ with finite
Lipschitz constant.
\end{thm}

Note that the hypothesis of this theorem are slightly weaker than
the ones of Bandt, since $r = \di\sup_{i\in\N}
\text{Lip}(F_i)$ is not required to be strictly smaller than $1$. Moreover, if $D$ is bounded and $\sum_{i=1}^\infty \rho_i=1$ we have
$$\sum_i\rho_id(x_1,x_i)\leq\sum_i\rho_i \diam(D)=\diam(D)<\infty.$$
Hence, in the case of bounded $D$, for each (countable) probability
sequence we have a measure in $M_s(X)$
that is invariant with respect to $(\FF,\rho)$, (independently of
the value of $r$).

\begin{rem1}
In \cite{MM09} Mihail and Miculescu worked with Infinite Iterated Function System
with the same hypothesis than Bandt in \cite{Ban89} from a different viewpoint. They showed that for these IIFS, the unique invariant
set for $\mathcal{F}$ (which is bounded because of the hypothesis on the IIFS) coincides with the closure of the {\em canonical projection} of the shift space.

In this case, analogously to the finite case, one has immediately that the unique invariant measure is
$\pi_\sharp\tau$, and that the support  of this measure is the unique invariant set for $\mathcal{F}$ (here $\tau$ is
the product measure on $\mathbb{N}^\mathbb{N}$ induced by
$\rho(i)=\rho_i$ on each factor).

Their results rely strongly on the fact that the set of fixed points is bounded and the suppremum of the Lipschitz constants of the system is strictly smaller than $1$.
\end{rem1}

We will prove next that the support of any invariant measure for
$(\FF,\rho)$ must coincide with the smallest invariant set for
$\FF$, where $\rho$ is a probability
sequence, even for the case $r=1$.

We start proving that if $\mu$ is an invariant measure for
$(\FF,\rho)$ then its support is an invariant set for $\FF$. We
first need the following result:

\begin{lem} \label{r3} Let $F:X\to X$ be any Lipschitz map and let $\mu$ be a measure in $X$.
Then $F(\text{supp}~\mu)\subset \text{supp}~F_\sharp\mu$.
\end{lem}

\begin{proof}~

Let us consider $y=F(x)$ for some $x\in \text{supp}~\mu$. In order
to prove that $y\in \text{supp}~F_\sharp\mu$, we need to show that
any ball centered at $y$ has positive $F_\sharp\mu$-measure. We take
$\varepsilon>0$ and consider the ball $B(x,\delta) :=\{z\in X:
d(x,z)<\delta\}$, with $\delta=r^{-1}\varepsilon$ where $r=Lip(F)$.
Since $F$ is Lipschitz,
$$F(B(x,\delta))\subset B(F(x),r\delta)=B(y,\varepsilon),$$
and then,
$$B(x,\delta)\subset F^{-1}(B(y,\varepsilon)).$$
Hence,
$$F_\sharp\mu(B(y,\varepsilon))=\mu(F^{-1}(B(y,\varepsilon)))\geq\mu(B(x,\delta))>0,$$
because $x\in \text{supp}~\mu$. Thus $y\in \text{supp}~F_\sharp\mu$.
\end{proof}

We are now ready to prove the announced theorem.
\begin{thm} \label{b5} Let $(X,d)$ be complete metric space. Let
$\mathcal{F}=\{F_i:i\in\N\}$ be a countable family of contraction
maps on $X$ and let $\rho=\{\rho_i:i\in\N\}$ be a probability
sequence, i.e., $0<\rho_i<1$ and $\sum_{i=1}^\infty \rho_i=1$. If
$\mu$ is an invariant measure for $(\FF,\rho)$, then the support of
$\mu$ is an invariant set for $\mathcal{F}$.
\end{thm}

\begin{proof}~

Assume that $\mu$ is an invariant measure for $(\FF,\rho)$ and let
$A := \text{supp}\mu$. Then, by Definition~\ref{infinite-invariant},
$$\mu=\sum_{i=1}^\infty\rho_i {F_i}_\sharp\mu,$$
where ${F_i}_\sharp\mu$ is the measure defined by
${F_i}_\sharp\mu(E)=\mu(F_i^{-1}(E))$ for each $E\subset X$. By
Lemma \ref{r3} we have that $F_i(A)\subset
\text{supp}~{F_i}_\sharp\mu$ for all $i$. Further, it is clear that
$$\text{supp}~{F_i}_\sharp\mu\subset  \text{supp}~\left(\sum_{i=1}^\infty\rho_i{F_i}_\sharp\mu\right)= \text{supp}~\mu = A.$$
Hence, $\bigcup_{i=1}^\infty F_i(A)\subset A$. Since $A$ is closed,
we obtain $\overline{\bigcup_{i=1}^\infty F_i(A)}\subset A$.

On the other hand, let $a\in A$ and $\varepsilon>0$. Since
$\mu(B(a,\varepsilon))>0$ and
$\mu=\sum_{i=1}^\infty\rho_i{F_i}_\sharp\mu$, there must exist $i$ such
that $\mu(F_i^{-1}(B(a,\varepsilon)))>0$. Consequently
$$F_i^{-1}(B(a,\varepsilon))\cap A\neq\emptyset.$$
That is, there exists $x\in A$ such that $d(F_i(x),a)<\varepsilon$.
Then $a\in \overline{\bigcup_{i=1}^\infty F_i(A)}$. Thus, $A\subset
\overline{\bigcup_{i=1}^\infty F_i(A)}$.

The proof is complete.
\end{proof}

From this result and Theorem \ref{b7} one deduces that the support
of any invariant measure for $(\FF,\rho)$ contains the set $\overline{P}$, the closure of
the set of fixed points of finite compositions of members of $\FF$.
In the following theorem we prove that indeed the support of $\mu$ is equal to $\overline{P}$.

\begin{thm} Let $(X,d)$ be a complete metric space. Let
$\mathcal{F}=\{F_i:i\in\N\}$ be a countable family of contraction
maps on $X$ and let $\rho=\{\rho_i:i\in\N\}$ be a probability
sequence, i.e., $0<\rho_i<1$ and $\sum_{i=1}^\infty \rho_i=1$. If
$\mu$ is an invariant measure for $(\FF,\rho)$, then
$\text{supp}~\mu=\overline{P}$, where $P$ the set of fixed points of
finite compositions of members of $\FF$.
\end{thm}

\begin{proof}~

From Theorems  \ref{b5} and \ref{b7} we have that
$\overline{P}\subset \text{supp}~\mu$.

In order to prove the other inclusion, let us consider $x\not\in
\overline{P}$. We will prove that $x\not\in \text{supp}~\mu$, by
showing that there exists a neighbourhood of $x$ of zero
$\mu$-measure.

Since $x \not\in \overline{P}$, let $\varepsilon=d(x,\overline{P})/2>0$. Let
$$G=B(x,\varepsilon)=\{y\in X: d(x,y)<\varepsilon\}.$$
 We will
prove that $\mu(G)=0$.

Define, as before, the set $\overline{P}_\varepsilon$ as $\overline{P}_\varepsilon =\{y\in
X: d(y,\overline{P})<\varepsilon\}$. Notice that $G\cap
\overline{P}_{\varepsilon}=\emptyset$.

Now, let $i\in \mathbb{N}$ be arbitrary but fixed throughout the proof.
Since $F_i$ is a contraction map whose fixed point is $x_i$ and whose contraction factor is $r_i$, $\{F_i^k(x)\}_{
k\in\N}$ converges to $x_i$ and $\{r_i^k\}_{k\in\N}$ converges to 0.

Then there exists $k=k(i)$ such that
$r_i^k<1/2$ and $d(F_i^k(x),x_i)<\varepsilon/2$.

Further, again by the contractivity of $F_i$, we have that for any
$j\in \N$, $F_i^j(G)\subset
 B(F_i^j(x),r_i^j\varepsilon)$, in particular $F_i^k(G)\subset
 B(F_i^k(x),r_i^k\varepsilon)$.

Moreover, if $z\in
B(F_i^k(x),r_i^k\varepsilon)$ then
\begin{equation*}d(z,\overline{P})\leq d(z,x_i)\leq
d(z,F_i^k(x))+d(F_i^k(x),x_i)<r_i^k\varepsilon+\varepsilon/2<\varepsilon.
\end{equation*}
Therefore $F_i^k(G)\subset
\overline{P}_\varepsilon$.
Hence, $G\subset[(F_i^k)^{-1}(\overline{P}_{\varepsilon})\backslash
\overline{P}_{\varepsilon}]$.

To finish our claim, it will be enough  to prove that
$\mu\left((F_i^k)^{-1}(\overline{P}_{\varepsilon})\right)=\mu
(\overline{P}_{\varepsilon})$.

Since $\overline{P}$ is an invariant set, we have
$F_n(\overline{P}_{\varepsilon})\subset\overline{P}_{\varepsilon}$
for all $n\in \mathbb{N}$.
Consequently, for every $(i_1\dots i_k)\in \mathbb{N}^k$,
$F_{i_1\dots i_k}(\overline{P}_{\varepsilon})\subset
\overline{P}_{\varepsilon}$ and therefore
\begin{equation}\label{support}
\mu(\overline{P}_{\varepsilon})\leq \mu(F_{i_1\dots
i_k}^{-1}(\overline{P}_{\varepsilon})) \quad \text{ for all } \quad (i_1\dots i_k)\in
\mathbb{N}^k.
\end{equation}
Note that if for some $(i_1\dots i_k)\in \mathbb{N}^k$ we had a strict inequality in the last equation, by the invariance of $\mu$ we would have that
$$\mu(\overline{P}_{\varepsilon})=\sum_{i_1\dots i_k}\rho_{i_1}\dots\rho_{i_k}\mu(F_{i_1\dots i_k}^{-1}(\overline{P}_{\varepsilon}))\gneq\mu(\overline{P}_{\varepsilon}).$$
Therefore from equation~\eqref{support} we must have
$\mu(\overline{P}_{\varepsilon})= \mu(F_{i_1\dots
i_k}^{-1}(\overline{P}_{\varepsilon}))$ for all choices $(i_1\dots i_k)\in
\mathbb{N}^k$.

In particular, by taking $i_1=\dots=i_k=i$, we obtain
$\mu(\overline{P}_{\varepsilon})=
\mu((F_i^k)^{-1}(\overline{P}_{\varepsilon}))$, and the proof is
complete.
\end{proof}

As noted before, the existence of the invariant measure depends only
on the relatively weak condition $\sum_i\rho_id(x_1,x_i) < \infty$,
independently of the value of $r$. (For example if the set of fixed
points of $F_i$ is bounded, the condition is already satisfied and
guaranties the existence {\em and} uniqueness of an invariant
measure).

Our result shows, that in contrast to the case of {\em
invariant sets} which may not be unique, if an invariant measure
having bounded support exists, it is unique. Indeed, if $\mu$ is an invariant measure whose support is bounded,
from our result it follows that the set of fixed points of members of $\FF$ is bounded, which implies that $\sum_i\rho_id(x_1,x_i) < \infty$ and, consequently we obtain the uniqueness of the invariant measure.

\end{document}